\documentclass{amsart}

\usepackage{amssymb}
\usepackage{graphicx}
\usepackage{MnSymbol}
\usepackage{enumerate}

\usepackage{amsmath,amscd}

\usepackage{amsthm}

\title{Steep uncountable groups}

\author{Samuel M. Corson}

\author{Alexander Olshanskii}

\author{Olga Varghese}

\theoremstyle{definition}\newtheorem{theorem}{Theorem}

\theoremstyle{definition}

\theoremstyle{definition}
\theoremstyle{definition}
\theoremstyle{definition}
\theoremstyle{definition}
\theoremstyle{definition}
\theoremstyle{definition}\newtheorem{remark}[theorem]{Remark}
\theoremstyle{definition}
\theoremstyle{definition}\newtheorem{lemma}[theorem]{Lemma}
\theoremstyle{definition}
\theoremstyle{definition}
\theoremstyle{definition}
\theoremstyle{definition}
\theoremstyle{definition}

\newcommand{\Length}{\operatorname{Length}}

\begin{document}

\address{Matematika Saila, UPV/EHU, Sarriena S/N, 48940, Leioa - Bizkaia, Spain.}
\email{sammyc973@gmail.com}

\address{Department of Mathematics, Vanderbilt University, Nashville, Tennessee 37240.}
\email{alexander.olshanskiy@vanderbilt.edu}

\address{Institute of Mathematics, Heinrich-Heine-University D\"usseldorf, Universit\"atsstra{\upshape{\ss}}e 1, 40225, D\"usseldorf, Germany.}
\email{olga.varghese@hhu.de}

\thanks{The work of the first author was supported by the Basque Government Grant IT1483-22 and Spanish Government
Grants PID2019-107444GA-I00 and PID2020-117281GB-I00, and also by the Heilbronn Institute for Mathematical Research and the Fields Institute.  The work of the second author was supported in part by the NSF grant DMS 1901976.  The work of the third author was funded by DFG grant VA 1397/2-2.}

\begin{abstract}  We produce a simple group $G$ of cardinality $\aleph_1$ which is Artinian (every strictly descending chain of subgroups is finite), satisfies a Burnside law and such that for each uncountable subset $Y \subseteq G$ there exists a natural number $n_Y$ for which every element of $G$ may be expressed as a product of length at most $n_Y$ of elements in $Y^{\pm 1}$.  

In particular this group is J\'onsson (every proper subgroup is of strictly smaller cardinality) and strongly bounded (every abstract action on a metric space has bounded orbits); this is the first example of an uncountable group having both of these properties which is constructed without using the continuum hypothesis.  The group $G$ can also be made so that all subgroups are simple and all nontrivial subgroups are malnormal in $G$.

\end{abstract}

\maketitle

\begin{section}{Introduction}

We construct uncountable groups having unusual properties.  The main theorem is the following.

\begin{theorem}\label{maintheorem}  There exists a simple group $G$ such that

\begin{enumerate}

\item $|G| = \aleph_1$;

\item for any uncountable $Y \subseteq G$ there exists a natural number $n_Y \in \omega$ such that each element of $G$ can be written as a product, of length at most $n_Y$, of elements in $Y^{\pm 1}$;

\item each subgroup of $G$ is cyclic or simple; 

\item if $K \leq G$ is not a proper subgroup of a cyclic subgroup of $G$ then $K$ is malnormal, that is $K \cap gKg^{-1}=\{1_G\}$ for every $g\in G \setminus K$; and

 \item $G$ can be taken

\begin{enumerate}[(a)]

\item to satisfy a Burnside law $x^M = 1_G$ (where $M$ is a large enough odd integer) and be Artinian (every strictly descending chain of subgroups is finite), or

\item to be torsion-free and every strictly descending chain of non-cyclic subgroups is finite.

\end{enumerate}
\end{enumerate}
\end{theorem}

The group given by (5) (a) is ``steep'' in its upward and downward behavior: from any uncountable subset the group will be generated in only finitely many steps, and one can spend only finitely many steps passing into proper subgroups.  We can also take the exponent $M$ in (5) (a) to be a prime, so that every subgroup of $G$ is simple by (3) and every nontrivial subgroup of $G$ is malnormal by (4).  A group of cardinality $\aleph_1$ has property (2) above if and only if it satisfies both of the following two properties (see Lemma \ref{stronglyboundedequiv}).  A group $H$ is

\begin{itemize}

\item \emph{J\'onsson} if every proper subgroup of $H$ is of strictly smaller cardinality than $|H|$;

\item \emph{strongly bounded} if every abstract action of $H$ by isometries on a metric space has bounded orbits.

\end{itemize}

\noindent Obraztsov constructed a simple J\'onsson group satisfying (1) and (5) (a) (see \cite{Ob000} or \cite[Corollary 35.4]{Ol}), and properties (3) and (4) can also be checked for his construction.  We shall use a modification of his argument to obtain our results.

In \cite{Sh} Shelah used the continuum hypothesis to construct a torsion-free simple group $\mathcal{S}$ of cardinality $\aleph_1$ for which there exists a natural number $\mathfrak{n}$ so that given any uncountable $Y \subseteq \mathcal{S}$ every element $s \in \mathcal{S}$ may be written as a product of length at most $\mathfrak{n}$ of elements of $Y$.  Note that inverses of elements of $Y$ are not used (his group is a \emph{J\'onsson semigroup}).  Recent innovations of Banakh \cite{Ban} have reduced the constant $\mathfrak{n}$ drastically, while still using the continuum hypothesis.  The comparable torsion-free result of our paper is the variation (5) (b).  Our proof only uses the standard ZFC axioms (without the continuum hypothesis) while incurring two costs:  the number $n_Y$ depends on $Y$ and our generation of the group also utilizes inverses of elements in the set $Y$.  In the torsion variation (5) (a) the group is a J\'onsson semigroup since $g^{-1} = g^{M-1}$.

Shelah also gave a J\'onsson group of cardinality $\aleph_1$ in \cite{Sh} whose construction uses only ZFC.  It was not known until recently \cite{CoSh} whether there could be strongly bounded groups of cardinality $\aleph_1$ constructed using only ZFC.  Our groups are the first known cardinality $\aleph_1$ examples which are both J\'onsson and strongly bounded produced from ZFC alone.

We point out that there is no countably infinite group satisfying the analogue of property (2):  it is an immediate consequence of  \cite[Lemma 5]{Pr} that for every countably infinite group $G$ there exists an infinite subset $Y \subseteq G$ such that for every $n \in \omega$ there is an element $g \in G$ which cannot be written as a word in $Y^{\pm 1}$ of length at most $n$.  This difference can be seen as a group-theoretic reflection of the Todor\v{c}evi\'c strong coloring theorem $\aleph_1 \not\rightarrow [\aleph_1]^2_{\aleph_1}$ and Ramsey's theorem $\aleph_0 \rightarrow (\aleph_0)^2_2$ (see respectively \cite{Tod} and \cite[Theorem 9.1]{J}) and indeed these combinatorics theorems are used in the proofs of the respective uncountable and countable behavior.

\end{section}

\begin{section}{Proof of the Main Theorem}

We justify the connection between property (2) of Theorem \ref{maintheorem} and the conjunction of being strong bounded and J\'onsson, and afterwards we prove Theorem \ref{maintheorem}.  If $H$ is a group and $Y \subseteq H$ we let

\begin{center}
$Y^{[n]} := \{g_1^{\varepsilon_1} \cdots g_{m}^{\varepsilon_{m}} \mid 0 \leq m \leq n \text{ and } (\forall 1 \leq i \leq m)[ g_i \in Y \text{ and }\varepsilon_i \in \{-1, 1\}]\}$
\end{center}

\noindent and

\begin{center}
$\mathcal{G}(Y) := Y^{[2]}$.
\end{center}

\noindent The following is \cite[Proposition 2.7]{dC}.

\begin{lemma}\label{Cornulier}  A group $H$ is strongly bounded if and only if for any sequence $\{Y_n\}_{n \in \omega}$ of subsets of $H$ such that $\mathcal{G}(Y_n) \subseteq Y_{n + 1}$ and $\bigcup_{n \in \omega} Y_n = H$ we have $Y_N = H$ for some $N \in \omega$.

\end{lemma}

As is standard, an ordinal is the set of all ordinals which are below it, so for example $3 = \{0, 1, 2\}$ and $\omega + 1 = \{0, 1, \ldots, \omega\}$ and the ordering on an ordinal coincides precisely with the membership relation $\in$.  A cardinal is the smallest ordinal of its cardinality, so $\aleph_1$ is the set of all countable ordinals.  A subset $X$ of an ordinal $\alpha$ is \emph{unbounded} if there is no $\beta \in \alpha$ such that $\beta$ is strictly above all elements in $X$.  Recall that the \emph{cofinality} of an ordinal $\alpha$ is the smallest cardinality of an unbounded subset of $\alpha$ (some examples: the cofinality of a successor $\alpha + 1$ is $1$ as the set $\{\alpha\}$ is unbounded in $\alpha + 1$, and the cofinality of $\aleph_1$ is $\aleph_1$).

\begin{lemma}\label{stronglyboundedequiv}  Suppose that $H$ is a group with $|H|$ of uncountable cofinality.  The following are equivalent:

\begin{enumerate}

\item $H$ is strongly bounded and J\'onsson;

\item for any $Y \subseteq H$ with $|Y| = |H|$ there exists some $n_Y \in \omega$ such that $H = Y^{[n_Y]}$.

\end{enumerate}

\end{lemma}

\begin{proof} (1) $\Rightarrow$ (2)  Let $Y \subseteq H$ with $|Y| = |H|$.  Since $H$ is J\'onsson and $|H| = |Y| \leq |\langle Y \rangle| \leq |H|$ we see that $\langle Y \rangle = H$.  For each $n \in \omega$ we let $Y_n = Y^{[2^n]}$ and notice that $\bigcup_{n \in \omega} Y_n = \langle Y \rangle = H$ and $\mathcal{G}(Y_n) \subseteq Y_{n + 1}$.  Since $H$ is strongly bounded there exists some $N \in \omega$ for which $Y_N = H$ by Lemma \ref{Cornulier}, so $Y^{[2^N]} = H$.

(2) $\Rightarrow$ (1)  We let $L \leq H$ with $|L| = |H|$.  Then there exists some $n_L \in \omega$ such that $H = L^{[n_L]}$, but since $L$ is a group we see that $L = L^{[n_L]} = H$ and so $H$ is J\'onsson.  Next we let $\{Y_n\}_{n\in \omega}$ be a sequence of subsets of $H$ with $\mathcal{G}(Y_n) \subseteq Y_{n + 1}$ and $\bigcup_{n\in \omega} Y_n = H$.  Since $|H|$ is of uncountable cofinality there must exist some $m \in \omega$ with $|Y_m| = |H|$.  Now select $n_{Y_m} \in \omega$ for which $Y_m^{[n_{Y_m}]} = H$ and notice that $H = Y_m^{[n_{Y_m}]} \subseteq Y_{m + n_{Y_m} + 1}$, and $H$ is seen to be strongly bounded by Lemma \ref{Cornulier}.

\end{proof}

The principal innovation required to prove the main theorem is the following embedding lemmas (cf. \cite[Theorem 35.1]{Ol} and \cite[Lemma 2.4]{CoSh}), whose proof we give at the end of the paper.  Lemma \ref{correct} is used in providing the group satisfying (5) (a) and Lemma \ref{correct2} is used for (5) (b).

\begin{lemma}\label{correct}  There is a group word $W(x_0, x_1, y)$ such that the following holds.  If $M$ is a sufficiently large odd number (say $M > 10^{75}$), $L$ is a nontrivial countable group satisfying the law $x^M = 1$ and $f:  (L \setminus \{1_L\})^2 \rightarrow L \setminus \{1_L\}$ is a function, then there exists a countably infinite simple group $H$ and $c\in H$ such that

\begin{enumerate}[(a)]

\item $L \leq H$;

\item $c \in H\setminus L$;

\item $H = \langle L \cup \{c\}\rangle$;

\item for all $(g_0, g_1) \in (L \setminus \{1_L\})^2$ we have $W(g_0, g_1, c) = f(g_0, g_1)$;

\item if $h_0 \in L \setminus \{1_L\}$ and $h_1 \in H \setminus L$ then $H = \langle h_0, h_1 \rangle$;

\item every proper subgroup of $H$ is either cyclic of order dividing $M$ or is conjugate in $H$ to a subgroup of $L$; and

\item if $h \in H \setminus L$ then $L \cap hLh^{-1} = \{1_H\}$ and if $h \in H$ is not conjugate to an element of  $L$ then the centralizer in $H$ of $h$ is cyclic and malnormal in $H$.

\end{enumerate}

\end{lemma}

\begin{lemma}\label{correct2}  There is a group word $W(x_0, x_1, y)$ such that the following holds.  If $L$ is a nontrivial torsion-free countable group and $f:  (L \setminus \{1_L\})^2 \rightarrow L \setminus \{1_L\}$ is a function, then there exists a countably infinite simple torsion-free group $H$ and $c\in H$ such that

\begin{enumerate}[(a)]

\item $L \leq H$;

\item $c \in H\setminus L$;

\item $H = \langle L \cup \{c\}\rangle$;

\item for all $(g_0, g_1) \in (L \setminus \{1_L\})^2$ we have $W(g_0, g_1, c) = f(g_0, g_1)$;

\item if $h_0 \in L \setminus \{1_L\}$ and $h_1 \in H \setminus L$ then $H = \langle h_0, h_1 \rangle$;

\item every proper subgroup of $H$ is either cyclic or is conjugate in $H$ to a subgroup of $L$; and

\item if $h \in H \setminus L$ then $L \cap hLh^{-1} = \{1_H\}$ and if $h \in H$ is not conjugate to an element of  $L$ then the centralizer in $H$ of $h$ is cyclic and malnormal in $H$.

\end{enumerate}

\end{lemma}

\begin{remark}\label{niceArtinian}  Notice that by letting $L$ be a cyclic group of order $M$ we can apply Lemma \ref{correct} to obtain a countably infinite group $H$ satisfying the law $x^M = 1$, such that every proper subgroup is cyclic of order dividing $M$, and every proper subgroup $K$ in $H$ which is not a proper subgroup of a cyclic subgroup is malnormal.  In particular, $H$ is a simple infinite Artinian group satisfying $x^M = 1$.

\end{remark}

What follows is a result of Todor\v{c}evi\'c.

\begin{lemma}\label{thenicefunction}(\cite[$\mathsection$ 4]{Tod})  If $E$ is a set of cardinality $\aleph_1$ then there exists a function $J: E \times E \rightarrow E$ such that for any uncountable $Z \subseteq E$ we have $J(Z \times Z) = E$.
\end{lemma}

\begin{proof}[Proof of Theorem \ref{maintheorem}]  We will first produce the group $G$ satisfying (5) (a) in the statement of the theorem.  Fix an odd natural number $M$ as in Lemma \ref{correct}.  Take $E$ to be a set of cardinality $\aleph_1$ and let $E = \{e_{\delta}\}_{\delta \in \aleph_1}$ be an enumeration of its elements.  Let $J: E \times E \rightarrow E$ be as in Lemma \ref{thenicefunction}.  We define a collection $\{E_{\alpha}\}_{\alpha \in \aleph_1}$ of subsets of $E$ such that

\begin{itemize}

\item $E = \bigcup_{\alpha \in \aleph_1} E_{\alpha}$;

\item $E_{\alpha} \subseteq E_{\alpha'}$ when $\alpha < \alpha'$;

\item $E_{\alpha}$ is countably infinite for each $\alpha  \in \aleph_1$;

\item $E_{\alpha + 1} \setminus E_{\alpha}$ is countably infinite for each $\alpha \in \aleph_1$;

\item $E_{\alpha} = \bigcup_{\beta < \alpha} E_{\beta}$ whenever $\alpha \neq 0$ is a limit ordinal; and

\item $J(E_{\alpha} \times E_{\alpha}) \subseteq E_{\alpha}$ for each $\alpha \in \aleph_1$.

\end{itemize}

\noindent We let $E_{0, 0} = \{e_{\delta}\}_{\delta \in \omega}$ and take $E_{0, n+1} = E_{0, n} \cup J(E_{0, n} \times E_{0, n})$ for each $n \in \omega$.  Take $E_0 = \bigcup_{n \in \omega} E_{0, n}$.  All of the required properties so far have been satisfied.  If $\alpha \in \aleph_1 \setminus \{0\}$ is a limit ordinal then we let $E_{\alpha} = \bigcup_{\beta < \alpha} E_{\beta}$.  If $E_{\alpha}$ has been defined and $E_{\alpha + 1}$ has not yet been defined, then as $E_{\alpha}$ is countable we select $\epsilon_{\alpha} \in \aleph_1$ with $\epsilon_{\alpha} > \alpha$ and such that $\epsilon_{\alpha} \setminus \sup \{\delta \in \aleph_1 \mid e_{\delta} \in E_{\alpha}\}$ is infinite.  Set $E_{\alpha + 1, 0} = \{e_{\delta}: \delta < \epsilon_{\alpha}\}$ and take $E_{\alpha + 1, n+1} = E_{\alpha + 1, n} \cup J(E_{\alpha + 1, n} \times E_{\alpha + 1, n})$ for each $n \in \omega$, and $E_{\alpha + 1} = \bigcup_{n \in \omega} E_{\alpha + 1, n}$.  The check that the properties have been satisfied is straightforward.

Let us add a new element to $E$, take $x_0 \notin E$ and define $X$ to be the disjoint union $E \sqcup \{x_0\}$ and for each $\alpha \in \aleph_1$ let $X_{\alpha} = E_{\alpha} \sqcup \{x_0\}$.  We will inductively define a group structure on $X$ via group structures $G_{\alpha}$ on $X_{\alpha}$ in such a way that $G_{\alpha}$ is a subgroup of $G_{\alpha'}$ whenever $\alpha < \alpha'$.  By Remark \ref{niceArtinian} let $X_0$ be given a group structure $G_0$ so that $G_0$ is a countably infinite simple Artinian group satisfying the group law $x^M = 1$, and the element $x_0$ corresponds to the identity element of $G_0$, such that any proper subgroup of $G_0$ is .

Suppose we have already given group structures $G_{\beta}$ to each $X_{\beta}$ with $\beta < \alpha$ in such a way that $G_{\beta}$ is a subgroup of $G_{\beta'}$ whenever $\beta < \beta'$.  If $\alpha \neq 0$ is a limit ordinal then we give $X_{\alpha}$ the unique group structure $G_{\alpha}$ so that $G_{\beta}$ is a subgroup of $G_{\alpha}$ when $\beta < \alpha$.  If $\alpha$ is a successor ordinal, say $\alpha = \gamma + 1$, then select an element $c_{\gamma} \in X_{\alpha} \setminus X_{\gamma}$ and as $X_{\alpha} \setminus X_{\gamma}$ is countably infinite we apply Lemma \ref{correct} to endow $X_{\alpha}$ with a group structure $G_{\alpha}$ such that

\begin{enumerate}[(i)]

\item $G_{\gamma}$ is a subgroup of $G_{\alpha}$;

\item $G_{\alpha} = \langle G_{\gamma} \cup \{c_{\gamma}\}\rangle$;

\item for all $g_0, g_1 \in G_{\gamma} \setminus \{x_0\}$ we have $W(g_0, g_1, c_{\gamma}) = J(g_0, g_1)$;

\item if $g_0 \in G_{\gamma} \setminus \{x_0\}$ and $g_1 \in G_{\alpha} \setminus G_{\gamma}$ then $G_{\alpha} = \langle g_0, g_1 \rangle$;

\item the group $G_{\alpha}$ satisfies the law $x^M = 1$ and every proper subgroup of $G_{\alpha}$ is either cyclic of order dividing $M$ or is conjugate to a subgroup of $G_{\gamma}$;

\item the group $G_{\alpha}$ is simple; and

\item if $g \in G_{\alpha} \setminus G_{\gamma}$ then $G_{\gamma} \cap gG_{\gamma}g^{-1} = \{1_{G_{\alpha}}\}$ and if $g$ is not conjugate to an element in $G_{\gamma}$ then the centralizer of $g$ is cyclic and malnormal in $G_{\alpha}$.
\end{enumerate}

\noindent Finally, we endow $X = \bigcup_{\alpha \in \aleph_1} X_{\alpha}$ with the unique group structure $G$ which is given by the group structures $G_{\alpha}, \alpha < \aleph_1$, so that each $G_{\alpha}$ is a subgroup of $G$.

We must check that the group $G$ satisfies the desired properties.  Let $Y \subseteq G$ be uncountable and without loss of generality $Y \cap G_0 = \emptyset$.  Define function $Q: Y \rightarrow \aleph_1$ by letting $Q(g) = \min\{\alpha <\aleph_1 \mid g \notin G_{\alpha}\}$ (the set of which we are taking a minimum is nonempty since we have excluded elements of $G_0$ from $Y$).  By how the group structure on $G_{\alpha}$ was defined when $\alpha$ is a limit ordinal we see that $Q(y)$ is always a successor ordinal.  Also, as $Y$ is an uncountable subset of $X$ and $X_{\alpha}$ was countable for each $\alpha \in \aleph_1$, we see that $Q(Y)$ is uncountable (and therefore unbounded in $\aleph_1$).

We notice that for every $\alpha \in Q(Y) \setminus \min (Q(Y))$ we have some natural number $k_{\alpha}$ such that $c_{\alpha} \in Y^{[k_{\alpha}]}$.  To see this, we fix $\alpha \in Q(Y) \setminus \min(Q(Y))$ and select $g_0 \in Y$ with $Q(g_0)$ minimal and $g_1 \in Y$ such that $Q(g_1) > \alpha + 1$.  We have $c_{\alpha} \in G_{\alpha + 1} \leq G_{Q(g_1) - 1} = \langle g_0, g_1 \rangle$ by condition (iv) above.  Since $\langle g_0, g_1 \rangle = \bigcup_{k \in \omega} \{g_0, g_1\}^{[k]}$ we see that $c_{\alpha} \in \{g_0, g_1\}^{[k]} \subseteq Y^{[k]}$ for some $k = k_{\alpha} \in \omega$.

Next we notice that $Y^{[n]} \cap \{c_{\alpha}\}_{\alpha \in Q(Y)}$ is uncountable for some $n \in \omega$.  To see this, note that $\bigcup_{n \in \omega} Y^{[n]} \supseteq \{c_{\alpha}\}_{\alpha \in Q(Y) \setminus \min(Q(Y))}$ by the argument in the preceeding paragraph.  Were it the case that $Y^{[n]} \cap \{c_{\alpha}\}_{\alpha \in Q(Y)}$ is countable for every $n$ then we would have $\{c_{\alpha}\}_{\alpha \in Q(Y) \setminus \min(Q(Y))}$ countable, as a countable union of the countable sets $Y^{[n]} \cap \{c_{\alpha}\}_{\alpha \in Q(Y) \setminus \min(Q(Y))}$, but this is absurd.  Thus $Y^{[n]} \cap \{c_{\alpha}\}_{\alpha \in Q(Y)}$ is uncountable for some $n \in \omega$ as required.

Now since $Y^{[n]} \cap \{c_{\alpha}\}_{\alpha \in Q(Y)}$ is uncountable we let $g \in G$ be given.  If $g = 1_G = x_0$ then it is clear that $g \in Y^{[1]}$, so we may assume that $g \in G \setminus \{1_G\}$.  Select $c_{\alpha_0}, c_{\alpha_1} \in Y^{[n]} \cap \{c_{\alpha}\}_{\alpha \in Q(Y)}$ such that $J(c_{\alpha_0}, c_{\alpha_1}) = g$.  Select $c_{\alpha_2} \in Y^{[n]} \cap \{c_{\alpha}\}_{\alpha \in Q(Y)}$ with $\alpha_2 >  \alpha_0, \alpha_1$ and we have $W(c_{\alpha_0}, c_{\alpha_1}, c_{\alpha_2}) = g$.  Thus letting $n_Y = n\cdot\Length(W)$ we have obtained the required constant.  Thus conditions (1) and (2) of the theorem hold.  In particular, $G$ is a J\'onsson group.

\begin{remark} \label{properissmall} We point out that if $\alpha < \aleph_1$ is a limit ordinal and $K$ is a proper subgroup of $G_{\alpha}$, then $K$ is a subgroup of $G_{\beta}$ for some $\beta < \alpha$.  Supposing for contradiction that this fails, we let $g \in G_{\alpha}$ be given and select $\beta_0 < \alpha$ and $g_0, g_1 \in K$ such that $g \in G_{\beta_0}$, $g_0 \in G_{\beta_0} \setminus \{1_G\}$ and $g_1 \in K \setminus G_{\beta_0}$.  Then we have $g \in G_{\beta_0} \leq \langle g_0, g_1 \rangle \subseteq K$ by (iv).  Thus $K = G_{\alpha}$ as $g$ was arbitrary, contradicting the assumption that $K$ is a proper subgroup of $G_{\alpha}$.
\end{remark}

The group $G$ is simple as a direct limit of simple groups.  That each subgoup of $G$ is either simple or cyclic is proven by induction.  We have that $G_0$ is simple and every proper subgroup of $G_0$ is cyclic.  Suppose that for all $\beta < \alpha$ every subgroup of each $G_{\beta}$ is either simple or cyclic.  If $\alpha < \aleph_1$ is such that $\alpha = \gamma + 1$, then by condition (v) we know that all subgroups of $G_{\alpha}$ are either cyclic or are conjugate to a subgroup of $G_{\beta}$ (and hence either simple or cyclic).  If $\alpha < \aleph_1$ is a limit ordinal then we know by condition (vi) that $G_{\alpha}$ is simple, and by Remark \ref{properissmall} any proper subgroup of $G_{\alpha}$ is a subgroup of some $G_{\beta}$, with $\beta < \alpha$, we see that the claim holds for $G_{\alpha}$ in this case as well.  Thus for all $\alpha < \aleph_1$, each subgroup of $G_{\alpha}$ is either cyclic or simple.  As we know that $G$ is simple, and by (1) and (2) any proper subgroup $K$ of $G$ is countable and therefore a subgroup of some $G_{\alpha}$, we conclude that any subgroup of $G$ is either cyclic or simple and property (3) holds.

To see that property (4) holds we first prove by induction that: if $K$ is a proper subgroup of $G_{\alpha}$ and not a proper subgroup of a cyclic subgroup of $G_{\alpha}$ then $K$ is malnormal in $G_{\alpha}$.  This holds for $G_0$ by design.  Suppose that the claim holds for $G_{\beta}$ for each $\beta < \alpha$.  Suppose that $\alpha = \gamma + 1$ and that $K$ is a proper subgroup of $G_{\alpha}$ and not a proper subgroup of a cyclic subgroup of $G_{\alpha}$.  If $K$ is conjugate to a proper subgroup of $G_{\beta}$ then $K$ is malnormal in $G_{\alpha}$ since $G_{\gamma}$ is malnormal in $G_{\alpha}$ (by (vii)) and $K$ is malnormal (by induction) in $G_{\gamma}$.  If $K$ is conjugate in $G_{\alpha}$ to the subgroup $G_{\gamma}$ then $K$ is malnormal in $G_{\alpha}$ since $G_{\gamma}$ is malnormal in $G_{\alpha}$ (by (vii)).  By (v) the only other alternative is that $K$ is cyclic, say generated by $h$, and $h$ is not conjugate to an element in $G_{\gamma}$, and in this case we know the centralizer of $h$ includes $K$ (therefore is equal to $K$) and is malnormal in $G_{\alpha}$ by (vii).  Supposing that $\alpha < \aleph_1$ is a limit ordinal and $K$ is a proper subgroup of $G_{\alpha}$ which is not a proper subgroup of a cyclic subgroup of $G_{\alpha}$, we have by Remark \ref{properissmall} that $K$ is a subgroup of $G_{\beta}$ with $\beta < \alpha$.  Then $K$ is malnormal in $G_{\beta'}$ for all $\beta \leq \beta' < \alpha$, which implies that $K$ is malnormal in $G_{\alpha}$.  Thus we have shown the inductive claim.  Certainly $G$ is malnormal as a subgroup of $G$.  If, now, $K$ is a proper subgroup of $G$ and not a proper subgroup of a cyclic subgroup of $G$, then as $G$ is J\'onsson we have that $K$ is countable and a proper subgroup of $G_{\alpha}$ for some $\alpha < \aleph_1$.  Then by induction $K$ is malnormal in $G_{\alpha'}$ for all $\alpha \leq \alpha' <\aleph_1$ and so $K$ is malnormal in $G$.  Thus property (4) holds.

The law $x^M = 1$ holds in $G$ since it holds in each of the $G_{\alpha}$.

We check by induction that each of the groups $G_{\alpha}$ is Artinian.  The group $G_{0}$ was selected to be Artinian.  If $\alpha = \gamma + 1$ and $G_{\gamma}$ is Artinian then by construction the group $G_{\gamma + 1}$ is also Artinian, since any proper subgroup is either cyclic of order dividing $M$ or is conjugate to a subgroup of the Artinian group $G_{\gamma}$.  If $\alpha < \aleph_1$ is a limit ordinal then by Remark \ref{properissmall} any proper subgroup $K$ of $G_{\alpha}$ is a subgroup of $G_{\beta}$ for some $\beta < \alpha$, so $G_{\alpha}$ is Artinian.  Thus each $G_{\alpha}$ is Artinian.  Since $G$ is J\'onsson we know that any proper subgroup $K$ of $G$ must be countable, and so $K$ is a subgroup of an Artinian group $G_{\alpha}$.  Thus $G$ is itself Artinian.  Thus we have (5) (a) and the proof in this case is complete.

The torsion-free construction in case (5) (b) is comparable.  We give $G_0$ an infinite cyclic group structure and utilize Lemma \ref{correct2} instead of Lemma \ref{correct} at each successor stage, replacing hypothesis (v) with

\begin{enumerate}

\item[(v$^*$)] the group $G_{\alpha}$ is torsion-free and every proper subgroup of $G_{\alpha}$ is either cyclic or is conjugate to a subgroup of $G_{\gamma}$.

\end{enumerate}

\noindent The check that the obtained group $G$ is simple and has properties (2), (3), and (4) is unchanged, and that $G$ is torsion-free follows from the fact that it is the increasing union of torsion-free groups.  That every strictly descending chain of non-cyclic subgroups is finite is seen by a minor modification to the check of the Artinian property in (5) (a).

\end{proof}

It remains to prove Lemmas \ref{correct} and \ref{correct2}.  These embedding lemmas are proved by giving minor modifications to the beautiful proof of \cite[Theorem 35.1]{Ol} originally obtained by Obraztsov, which was a modification of the proof of \cite[Theorem 28.1]{Ol}.  We shall assume familiarity with \cite[Chapter 11]{Ol}, especially $\mathsection$ 34 and $\mathsection$ 35, and utilize the notions and notation therein (including a full account of that machinery into the current paper would add a considerable number of pages).

\begin{proof}[Proof of Lemma \ref{correct}]  We begin by fixing positive real numbers

\begin{center}

$\delta >  \eta > \iota$

\end{center}

\noindent which are sufficiently small for standard graded diagram arguments, as well as integers $h = \delta^{-1}$, $d = \eta^{-1}$, $n = \iota^{-1}$, and $n_0 = M$ an odd integer with $n = \lfloor (h+1)^{-1} \cdot n_0\rfloor$, where $\lfloor r \rfloor$ denotes the largest integer less than or equal to $r$.  The word $W(x_0, x_1, y)$ is the following

$$W \equiv (zx_0)^{n + 3}x_1(zx_0)^{n + 7}x_1 \cdots x_1(zx_0)^{n + 4h - 1}.$$

Assume the hypotheses of the lemma.  We let $G_1$ be the group $L$ and $G_2$ be the finite cyclic group $\langle c \rangle_{M}$ of order $M$.  We define a presentation over the free product $G_1 * G_2$ satisfying condition $R$.  The relations are defined in the alphabet $\mathcal{A}$ which is the set of all nontrivial elements from $G_1$  and $G_2$. Below the definition is given by induction on rank $i\geq 0$ simultaneously with the definitions of simple words of rank $i$, periods of rank $i$, and groups $G(i)$.

There are neither periods nor relations in ranks 0 and 1.  (The reason is that we do not want to ruin $G_1$ and $G_2$.)  The groups $G(0)$ and $G(1)$ of ranks $0$ and $1$, respectively, are just the free product $G_1*G_2$.  The simple words in rank $i=0$, $1$ are all the words which are not conjugate to powers of shorter words in rank $i$, i.e. in the group $G(i)$.  For example, every word $A \equiv cg$, where $g\in G_1 \setminus \{1\}$, has length $2$ and is simple in rank $1$ since it is not conjugate with an element of the free factors  $G_1$, $G_2$.

Assume that $i \geq 2$ and the above mentioned concepts are defined for ranks $\leq i-1$.  The set $\mathfrak{X}_i$ of periods of rank $i$ is defined as a set of simple in rank $i - 1$ words of length $i$, which is maximal with respect to the following property: if $A$, $B$ are different words from $\mathfrak{X}_i$, then A is not conjugate with $B^{\pm 1}$ in rank $i - 1$ (i.e. in $G(i - 1)$).  For example, all the words $A(g) \equiv cg$, $g\in G_1\setminus \{1\}$ can be included in $\mathfrak{X}_2$, since they are not conjugate in $G(1)$ for different $g$s.

Now we introduce the set $\mathcal{S}_i$ of new relations as follows.  First for each $A \in \mathfrak{X}_i$ we introduce the relation

$$
A^{n_0} = 1. \eqno{\text{(I)}}
$$

\noindent  These are \textit{relations of the first type} in the terminology of \cite[Chapter 11]{Ol}.  Next we introduce relations of the second type.  For each $A \in \mathfrak{X}_i$ we fix a maximal set $\mathcal{Y}_A$ so that

\begin{enumerate}

\item if $T \in \mathcal{Y}_A$ then $1 \leq \|T\| < d\|A\|$ (here $\| \cdot \|$ denotes word length in the alphabet $\mathcal{A}$), and

\item each double coset $\langle A \rangle, \langle A \rangle$ in $G(i - 1)$ contains at most one word in $\mathcal{Y}_A$, and this word is of minimal length among representatives of this coset.

\end{enumerate}

\noindent If $a_1 \notin \langle A \rangle \leq G(i - 1)$ then for $T \in \mathcal{Y}_A \setminus \langle A \rangle a_1 \langle A \rangle \subseteq G(i - 1)$ we add relation

$$
a_1TA^{n} T A^{n + 4} \cdots TA^{n + 4h - 4} = 1.  \eqno{\text{(II)}}
$$

\noindent If $a_2 \notin \langle A \rangle \leq G(i - 1)$ and $a_2 \notin \langle A \rangle a_1 \langle A \rangle$ then for $T \in \mathcal{Y}_A \setminus \langle A \rangle a_2 \langle A \rangle$ we add relation 

$$
a_2TA^{n + 1} T A^{n + 5} \cdots TA^{n + 4h - 3} = 1.  \eqno{\text{(III)}}
$$

\noindent Next, letting $s$ be the maximum number of letters occuring in the expression of $A$, if $a_{s+1}$ exists and $a_{s + 1} \notin \langle A \rangle \langle a_1, a_2 \rangle \langle A \rangle \subseteq G(i - 1)$ then for $T \in \mathcal{Y}_A \setminus \langle A \rangle a_{s + 1} \langle A \rangle$ we add relation

$$
a_{s + 1}TA^{n + 2} T A^{n + 6} \cdots TA^{n + 4h - 2} = 1.  \eqno{(\text{IV})}
$$

\noindent Finally if $i = 2$, for each $A \in \mathfrak{X}_2$ of form $A = cg$, with $g \in G_1 \setminus \{1\}$, and each $T = g' \in G_1 \setminus \{1\}$ we add relation

$$
(f(g, g'))^{-1}A^{n + 3} T A^{n + 7} \cdots TA^{n + 4h - 1} = 1  \eqno{\text{(V)}}
$$

\noindent where $f$ is the abstract function in the statement of the lemma.  We take $\mathcal{S}_i$ to be this new set of relations and write $\mathcal{R}_i = \mathcal{R}_{i - 1} \cup \mathcal{S}_i$.  The group $G(i)$ is defined as the factor group $G_1 * G_2/ \langle \langle \mathcal{R}_i \rangle \rangle$, and to complete the inductive definitions, we say that a simple word of rank $i \geq 2$ is a word which is not conjugate in $G(i)$ to a power of a shorter word and is not conjugate to any power of a period of rank $j \leq i$.

The sequence of relations $\{\mathcal{R}_i\}_{i \in \omega}$ will satisfy property $R$, proved along the same lines as \cite[Lemma 34.14]{Ol}.  The relations of form (I) - (IV) correspond to (2) - (5) in \cite[$\mathsection$ 34.3]{Ol}, with the exponents on the $A$ slightly altered to suit our purposes.  The proof that $H = G(\infty) = G_1 * G_1/\langle\langle \bigcup_{i \in \omega}\mathcal{R}_i\rangle\rangle$ is countably infinite and simple, and that conditions (e) and (f) hold, follows that of \cite[Theorem 35.1]{Ol}.  For property (g), if $h \in H \setminus L$ then $L \cap hLh^{-1} = \{1_H\}$ by \cite[Lemma 34.11]{Ol}.  If moreover $h$ is not conjugate to an element of $L$ then either $h$ is conjugate to a nontrivial element of $\langle c \rangle_{M}$, in which case without loss of generality $h$ is centralized by the group $\langle c \rangle_{M}$ which is malnormal in $H$ by \cite[Lemma 34.11]{Ol}, or $h$ is conjugate to a power of a period \cite[Lemma 34.7 (3)]{Ol} so without loss of generality $h$ is a nontrivial power of period $A$ and $\langle A \rangle$ clearly centralizes $h$ and is malnormal in $H$ by combining \cite[Lemma 34.7 (2) and (4)]{Ol} with \cite[Lemma 34.9]{Ol}.  Properties (a), (b), and (c) for the group $H$ are clear by construction and property (d) holds because of the relations of form (V).
\end{proof}

\begin{proof}[Proof of Lemma \ref{correct2}]  This follows by applying minor alterations to the proof of Lemma \ref{correct}.  Fix parameters $\delta >  \eta > \iota$ as before.  Assume the hypotheses of the lemma.  Let $G_1$ be the group $L$ and $G_2$ be the infinite cyclic group $\langle c \rangle_{\infty}$.  Let $G(0) = G(1) = G_1 * G_2$.  As in Lemma \ref{correct} we do not have periods or relators of ranks $0$ and $1$.  Similarly we define simple words in ranks $0$ and $1$ to be those words which are not conjugate to powers of shorter words in the group $G(0) = G(1) = G_1 * G_2$.

Assume we have defined periods of rank $j$, sets of relations $\mathcal{R}_j$, and simple words in rank $j$ for all $j < i$ and that $1 < i$.  We select the set of periods $\mathfrak{X}_i$ of rank $i$ precisely as was done in Lemma \ref{correct}.  Define the set $\mathcal{S}_i$ of new relations as before, except do not add any relations of the first type (i.e. those in (I)).  Write $\mathcal{R}_i = \mathcal{R}_{i - 1} \cup \mathcal{S}_i$ and $G(i) = G_1 * G_2/\langle\langle \mathcal{R}_i\rangle\rangle$.  It is clear that the elimination of relations of the first type in rank $i$ changes $G(i)$ and the content of the next inductive definitions, so all the sets of relations of the next ranks changes in comparison with the torsion case.  As before, define a simple word of rank $i$ to be a word which is not conjugate in $G(i)$ to a power of a shorter word and is not conjugate to any power of a period of rank $j \leq i$.

By the same argument as before, the sequence $\{\mathcal{R}_i\}_{i \in \omega}$ satisfies condition $R$ and $H = G(\infty) = G_1 * G_2/\langle\langle \bigcup_{i \in \omega} \mathcal{R}_i\rangle\rangle$ has the desired properties using the same word $W$.
\end{proof}

\end{section}

\end{document}